\documentclass[11pt,reqno]{article}

\usepackage{amsmath,amsthm,amsfonts,amssymb,latexsym,mathrsfs,color}
\usepackage{hyperref}

\newtheorem{theorem}{Theorem}
\newtheorem{corollary}[theorem]{Corollary}
\newtheorem{proposition}[theorem]{Proposition}

\newcommand{\des}{{\rm des\,}}

\newcommand{\mbn}{{\mathcal B}_n}

\newcommand{\msn}{{\mathcal S}_n}

\newcommand{\rz}{{\rm RZ}}

\newcommand{\lrf}[1]{\lfloor #1\rfloor}
\newcommand{\lrc}[1]{\lceil #1\rceil}

\title{A family of two-variable derivative polynomials for tangent and secant}

\author
{Shi-Mei Ma \footnote{ {\it Email address:}
shimeima@yahoo.com.cn (S.-M. Ma)} }
\date{\footnotesize Department of Information and Computing Science,
        Northeastern University at Qinhuangdao,\\ Hebei 066004,
        China}

\begin{document}

\maketitle

\begin{abstract}
In this paper we introduce a family of two-variable derivative polynomials for tangent and secant.
We study the generating functions for the coefficients of this family of polynomials. In particular, we establish a connection between these generating functions and Eulerian polynomials.
\end{abstract}
\section{Introduction}
Throughout this paper, we define $y=\tan(x)$ and $z=\sec(x)$. Denote by $D$ the differential operator ${d}/{d x}$.
Thus $D(y)=z^2$ and $D(z)=yz$. An important tangent identity is given by $$1+y^2=z^2.$$
In 1995, Hoffman~\cite{Hoffman95} considered two sequences of {\it derivative polynomials} defined respectively by
\begin{equation*}\label{derivapoly-1}
D^n(y)=P_n(y)\quad {\text and}\quad D^n(z)=z Q_n(y)
\end{equation*}
for $n\geq 0$.
From the chain rule it follows that the polynomials $P_n(u)$ satisfy $P_0(u)=u$ and $P_{n+1}(u)=(1+u^2)P_n'(u)$, and similarly $Q_0(u)=1$ and $Q_{n+1}(u)=(1+u^2)Q_n'(u)+uQ_n(u)$. The first few of the polynomials $P_n(u)$ are
\begin{align*}
P_1(u)&=1+u^2,\\
P_2(u)&=2u+2u^3,\\
P_3(u)&=2+8u^2+6u^4.
\end{align*}
Various refinements of the polynomials $P_n(u)$ and $Q_n(u)$ have been pursued by several authors,
see~\cite{Carlitz72,Cvijovic09,Franssens07,Hoffman99,Josuat10,Ma12} for instance.

Let $\msn$ denote the symmetric group of all permutations of $[n]$, where $[n]=\{1,2,\ldots,n\}$.
A permutation $\pi=\pi(1)\pi(2)\cdots\pi(n)\in\msn$
is {\it alternating} if $\pi(1)>\pi(2)<\cdots \pi(n)$. In other words, $\pi(i)<\pi({i+1})$ if $i$ is even and $\pi(i)>\pi({i+1})$ if $i$ is odd.
Let $E_n$ denote the number of alternating permutations in $\msn$. The number $E_n$ is called an {\it Euler number} because Euler considered the numbers $E_{2n+1}$. There has been a huge literature on Euler numbers (see~\cite{Stanley} for details).
In 1879, Andr\'e~\cite{Andre79} obtained that
$$y+z=\sum_{n=0}^\infty E_n\frac{x^n}{n!}.$$
Since the tangent is an odd function and the secant is an even function, we have
$$y=\sum_{n=0}^\infty E_{2n+1}\frac{x^{2n+1}}{(2n+1)!} \quad{\text and} \quad z=\sum_{n=0}^\infty E_{2n}\frac{x^{2n}}{(2n)!}.$$
For this reason the integers $E_{2n+1}$ are sometimes called the {\it tangent numbers} and the integers $E_{2n}$ are called the {\it secant numbers}.

Let $E(x)=y+z$. Clearly, $E(0)=1$.
It is easy to verify that
\begin{equation}\label{D-1}
2D(E(x))=1+E^2(x).
\end{equation}
Consider the derivative of~(\ref{D-1}) with respect to $x$, we have
\begin{equation}\label{D-2}
2^2D^2(E(x))=2E(x)+2E^3(x).
\end{equation}
By the derivative of~(\ref{D-2}) with respect to $x$, we get
$2^3D^3(E(x))=2+8E^2(x)+6E^4(x)$.
We now present a connection between $E(x)$ and $P_n(u)$
\begin{proposition}
For $n\geq 0$, we have $2^nD^n(E(x))=P_n(E(x))$.
\end{proposition}
\begin{proof}
It suffices to consider the case $n\geq 3$.
We proceed by induction on $n$. Assume that the statement is true for $n=k$.
Then
\begin{align*}
2^{k+1}D^{k+1}(E(x))&=2D(P_k(E(x)))\\
                    &=2P_k'(E(x))D(E(x))\\
                    &=(1+E^2(x))P_k'(E(x))\\
                    &=P_{k+1}(E(x)).
\end{align*}
Thus the statement is true for $k+1$, as desired.
\end{proof}

In~\cite{Ma12} we write the derivative polynomials in terms of $y$ and $z$ as follows:
\begin{equation*}
D^n(y)=\sum_{k=0}^{\lrf{({n-1})/{2}}}W_{n,k}y^{n-2k-1}z^{2k+2}\quad {\text and}\quad D^n(z)=\sum_{k=0}^{\lrf{{n}/{2}}}W_{n,k}^{\textit{l}}y^{n-2k}z^{2k+1} \quad {\text for}\quad n\geq 1.
\end{equation*}
In particular, we observed that the coefficients $W_{n,k}$ and $W_{n,k}^{\textit{l}}$ have simple combinatorial interpretations. The coefficient $W_{n,k}$ is the number of permutations in $\msn$ with $k$ interior peaks, where an interior peak of $\pi$ is an index $2\leq i\leq n-1$ such that $\pi(i-1)<\pi(i)>\pi(i+1)$.
The coefficient $W_{n,k}^{\textit{l}}$ is the number of permutations in $\msn$ with $k$ left peaks, where a left peak of $\pi$ is either an interior peak or else the index $1$ in the case $\pi(1)>\pi(2)$ (see~\cite{Petersen09} for instance).

In this paper we are concerned with a variation of the above definitions.
The organization of this paper is as follows. In Section~\ref{Section-2}, we collect some notation, definitions and results that will be needed in the rest of the paper. In Section~\ref{Section-3}, we establish a connection between
Eulerian numbers and the expansion of $(Dy)^n(y)$.
 In Section~\ref{Section-4}, we establish a connection between
Eulerian numbers of type $B$ and the expansion of $(Dy)^n(z)$.
In Section~\ref{Section-5}, we study some polynomials related to $(yD)^n(y)$ and $(yD)^n(z)$.
\section{Preliminaries}\label{Section-2}
For a permutation $\pi\in\msn$, we define a {\it descent} to be a position $i$ such that $\pi(i)>\pi(i+1)$. Denote by $\des(\pi)$ the number of descents of $\pi$. Let
\begin{equation*}
A_n(x)=\sum_{\pi\in\msn}x^{\des(\pi)+1}=\sum_{k=1}^nA(n,k)x^{k},
\end{equation*}
The polynomial $A_n(x)$ is called an {\it Eulerian polynomial}, while $A(n,k)$ is called an {\it Eulerian number}.
The exponential generating function for $A_n(x)$ is
\begin{equation}\label{Axz}
A(x,z)=1+\sum_{n\geq 1}A_n(x)\frac{t^n}{n!}=\frac{1-x}{1-xe^{t(1-x)}}.
\end{equation}
The numbers $A(n,k)$ satisfy the recurrence
\begin{equation}\label{recu-Euleriannum}
A(n+1,k)=kA(n,k)+(n-k+2)A(n,k-1)
\end{equation}
with the initial conditions $A(0,0)=1$ and $A(0,k)=0$ for $k\geq 1$ (see~\cite[A008292]{Sloane} for details).
The first few of the Eulerian polynomials $A_n(x)$ are
$$A_0(x)=1,A_1(x)=x,A_2(x)=x+x^2,A_3(x)=x+4x^2+x^3.$$
It is well known that
\begin{equation}\label{symmetric}
A_n(x)=x^{n+1}A_n\left(\frac{1}{x}\right).
\end{equation}
An explicit formula for $A(n,k)$ is given as follows:
\begin{equation*}
A(n,k)=\sum_{i=0}^k(-1)^i\binom{n+1}{i}(k-i)^n.
\end{equation*}

Let $B_n$ denote the set of signed permutations of $\pm[n]$ such that $\pi(-i)=-\pi(i)$ for all $i$, where $\pm[n]=\{\pm1,\pm2,\ldots,\pm n\}$.
Let
$${B}_n(x)=\sum_{k=0}^nB(n,k)x^{k}=\sum_{\pi\in \mbn}x^{\des_B(\pi)},$$
where
$$\des_B=|\{i\in[n]:\omega(i-1)>\omega({i+1})\}|$$
with $\pi(0)=0$.
The polynomial $B_n(x)$ is called an {\it Eulerian polynomial of type $B$}, while $B(n,k)$ is called an {\it Eulerian number of type $B$}.
Below are the polynomials ${B}_n(x)$ for $n\leq 3$:
$$B_0(x)=1,B_1(x)=1+x,B_2(x)=1+6x+x^2,B_3(x)=1+23x+23x^2+x^3.$$
The numbers $B(n,k)$ satisfy
the recurrence relation
\begin{equation}\label{Bnk-Euleriannum}
B(n+1,k)=(2k+1)B(n,k)+(2n-2k+3)B(n,k-1),
\end{equation}
with the initial conditions $B(0,0)=1$ and $B(0,k)=0$ for $k\geq 1$.
An explicit formula for $B(n,k)$ is given as follows:
\begin{equation*}
B(n,k)=\sum_{i=0}^k(-1)^{i}\binom{n+1}{i}(2k-2i+1)^{n}
\end{equation*}
for $0\leq k\leq n$ (see~\cite{Eriksen00} for details).

For $n\geq 0$, we always assume that
$$(Dy)^{n+1}(y)=(Dy)(Dy)^n(y)=D(y(Dy)^n(y)),$$
$$(Dy)^{n+1}(z)=(Dy)(Dy)^n(z)=D(y(Dy)^n(z)),$$
$$(yD)^{n+1}(y)=(yD)(yD)^n(y)=yD((yD)^n(y)),$$
$$(yD)^{n+1}(z)=(yD)(yD)^n(z)=yD((yD)^n(z)).$$

Clearly, $(Dy)^n(y+z)=(Dy)^n(y)+(Dy)^n(z)$.
For $n\geq 1$,
we define
\begin{equation*}\label{def-derivative-2}
(Dy)^n(y+z)=\sum_{k=0}^{2n}J(2n,k)y^{2n-k}z^{k+1}.
\end{equation*}
In Section~\ref{Section-3} and Section~\ref{Section-4}, we respectively obtain the following results:
$$J(2n,2k-1)=2^nA(n,k)  \quad {\text for}\quad 1\leq k\leq n,$$
and $$J(2n,2k)=B(n,k) \quad {\text for}\quad 0\leq k\leq n.$$
Let $J_n(x)=\sum_{k=0}^{2n}J(2n,k)x^k$ for $n\geq 1$.
Then $xJ_n(x)=2^nA_n(x^2)+xB_n(x^2)$.
Therefore, from~\cite[Theorem 3]{Ma12}, we have
\begin{equation}\label{JnxAnx}
xJ_n(x)=(1+x)^{n+1}A_n(x).
\end{equation}

Using~(\ref{JnxAnx}), we immediately obtain the following result.
\begin{proposition}
For $n\geq 1$, we have
$$(Dy)^n(y+z)=(y+z)^{n+1}\sum_{k=1}^nA(n,k)y^{n-k}z^k.$$
\end{proposition}
\section{On the expansion of $(Dy)^n(y)$}\label{Section-3}
For $n\geq 1$, we define
\begin{equation}\label{def-derivative-1}
(Dy)^n(y)=\sum_{k=1}^{n}E(n,k)y^{2n-2k+1}z^{2k}.
\end{equation}
\begin{theorem}\label{thm1}
For $1\leq k\leq n$, we have
$E(n,k)=2^nA(n,k)$, where $A(n,k)$ is an Eulerian number.
\end{theorem}
\begin{proof}
Note that $D(y^2)=2yz^2$. Then $E(1,1)=2A(1,1)$.
Since
$$(Dy)^{n+1}(y)=D(y(Dy)^n(y))=2\sum_{k=1}^nkE(n,k)y^{2n-2k+3}z^{2k}+
2\sum_{k=1}^n(n-k+1)E(n,k)y^{2n-2k+1}z^{2k+2},$$
there follows
\begin{equation}\label{T-recurrence}
E(n+1,k)=2(kE(n,k)+(n-k+2)E(n,k-1)).
\end{equation}
Comparison of~(\ref{T-recurrence}) with~(\ref{recu-Euleriannum}) gives the desired result.
\end{proof}

Let
\begin{equation}\label{Def-22}
F_n(y)=(Dy)^n(y)=\sum_{k=0}^nF(n,k)y^{2k+1}.
\end{equation}
By $F_{n+1}(y)=D(yF_n(y))$, we obtain
\begin{equation}\label{bny-recu}
F_{n+1}(y)=(1+y^2)F_n(y)+y(1+y^2)F'_n(y)
\end{equation}
with initial values $F_0(y)=y$. We define $F_n(y)=2^na_n(y)$ for $n\geq 0$.
In the following we present an explicit formula for $a_n(y)$.
\begin{theorem}\label{thm2}
For $n\geq 1$, we have
\begin{equation}\label{explicit-any}
a_n(y)=\sum_{k=1}^nA(n,k)y^{2n-2k+1}(1+y^2)^k.
\end{equation}
\end{theorem}
\begin{proof}
Combining~(\ref{def-derivative-1}) and~(\ref{Def-22}), we get
\begin{align*}
F_n(y)&=\sum_{k=1}^{n}E(n,k)y^{2k-1}(1+y^2)^{n-k+1}\\
      &=2^ny^{-1}(1+y^2)^{n+1}\sum_{k=1}^nA(n,k)\left(\frac{y^2}{1+y^2}\right)^k.
\end{align*}
It follows from~(\ref{symmetric}) that
\begin{align*}
a_n(y)&=y^{2n+1}\left(\frac{1+y^2}{y^2}\right)^{n+1}\sum_{k=1}^nA(n,k)\left(\frac{y^2}{1+y^2}\right)^k\\
      &=\sum_{k=1}^nA(n,k)y^{2n-2k+1}(1+y^2)^k.
\end{align*}
This completes the proof.
\end{proof}

Let $a_n(y)=\sum_{k=0}^na(n,k)y^{2k+1}$.
Equating the coefficients of $y^{2n-2k+1}$ on both sides of~(\ref{explicit-any}), we obtain
$$a(n,n-k)=\sum_{i=k}^n\binom{i}{k}A(n,i).$$
It follows from~(\ref{bny-recu}) that
\begin{equation}\label{ank}
a(n+1,k)=(k+1)a(n,k)+ka(n,k-1).
\end{equation}
We define $W_n(x)=\sum_{k=0}^na(n,k)x^{k+1}$. By~(\ref{ank}), we have
\begin{equation}\label{any-recurrence}
W_{n+1}(x)=(x+x^2)W'_n(x),
\end{equation}
with initial values $W_{0}(x)=x$.
The first few terms of $W_n(y)$ are given as follows:
\begin{align*}
W_1(x)&=x+x^2,\\
W_2(x)&=x+3x^2+2x^3,\\
W_3(x)&=x+7x^2+12x^3+6x^4,\\
W_4(x)&=x+15x^2+50x^3+60x^4+24x^5.
\end{align*}
The triangular array $\{a(n,k)\}_{n\geq 0, 0\leq k\leq n}$ is called a {\it Worpitzky triangle}, and it has been extensively studied by many authors (see~\cite[\textsf{A028246}]{Sloane}).

In view of~(\ref{any-recurrence}),
it is natural to consider the expansion of the operator $((x+x^2)D)^n$.
We define
\begin{equation}\label{def-Tnkx}
((x+x^2)D)^n=\sum_{k=1}^nG_{n,k}(x)(x+x^2)^kD^k \quad\textrm{for $n\ge 1$}¡£
\end{equation}
Applying the operator $(x+x^2)D$ on the left of~(\ref{def-Tnkx}), we get
\begin{equation}\label{Tnkx-1}
G_{n+1,k}(x)=k(1+2x)G_{n,k}(x)+(x+x^2)D(G_{n,k}(x))+G_{n,k-1}(x).
\end{equation}
On the other hand, since
$$D^k((x+x^2)D)=(x+x^2)D^{k+1}+k(1+2x)D^{k}+k(k-1)D^{k-1}.$$
Applying the operator $(x+x^2)D$ on the right of~(\ref{def-Tnkx}), we get
\begin{equation}\label{Tnkx-2}
G_{n+1,k}(x)=k(1+2x)G_{n,k}(x)+k(k+1)(x+x^2)G_{n,k+1}(x)+G_{n,k-1}(x).
\end{equation}
Comparison of~(\ref{Tnkx-1}) with~(\ref{Tnkx-2}) gives $D(G_{n,k}(x))=k(k+1)G_{n,k+1}(x)$.
Thus
\begin{equation*}\label{TnkxD}
G_{n,k}(x)=\frac{1}{k!(k-1)!}D^{k-1}(G_{n,1}(x)).
\end{equation*}
Thus $\deg G_{n,k}(x)=n-k$.
For $k=1$, set $G_{n}(x)=G_{n,1}(x)$.
Then~(\ref{Tnkx-1}) reduces to
\begin{equation*}\label{Tnx-recu}
G_{n+1}(x)=(1+2x)G_{n}(x)+(x+x^2)D(G_{n}(x))
\end{equation*}
with initial values $G_1(x)=1$.
We define $G_{n}(x)=\sum_{k=1}^nG(n,k)x^{k-1}$. It is easy to verify that
\begin{equation}\label{Tnk-recurrence}
G(n+1,k)=kG(n,k)+kG(n,k-1)
\end{equation}
with $G(1,1)=1$.
The {\it Stirling numbers of the second kind}, denoted by $S(n,k)$, may be defined by
the recurrence relation
\begin{equation}\label{recu-Stirling}
S(n+1,k)=kS(n,k)+S(n,k-1)
\end{equation}
with the initial conditions $S(0,0)=1$ and $S(n,0)=0$ for $n\geq 1$ (see~\cite[A008277 ]{Sloane} for details).
Comparison of~(\ref{Tnk-recurrence}) with~(\ref{recu-Stirling}) gives the following result.
\begin{proposition}
For $1\leq k\leq n$, we have
$G(n,k)=k!S(n,k)$.
\end{proposition}

\section{On the expansion of $(Dy)^n(z)$}\label{Section-4}
For $n\geq 0$,
we define
\begin{equation*}\label{def-derivative-2}
(Dy)^n(z)=\sum_{k=0}^{n}H(n,k)y^{2n-2k}z^{2k+1}.
\end{equation*}
\begin{theorem}\label{thm1}
For $0\leq k\leq n$, we have $H(n,k)=B(n,k)$, where $B(n,k)$ is an Eulerian numbers of type $B$.
\end{theorem}
\begin{proof}
Clearly, $H(0,0)=1$. Note that $D(yz)=y^2z+z^3$. Then $H(1,0)=B(1,0)$ and $H(1,1)=B(1,1)$.
Note that
$$(Dy)(Dy)^n(z)=\sum_{k=0}^n(1+2k)H(n,k)y^{2n-2k+2}z^{2k+1}+
\sum_{k=0}^n(2n-2k+1)H(n,k)y^{2n-2k}z^{2k+3}.$$
Thus we obtain
$$H(n+1,k)=(1+2k)H(n,k)+(2n-2k+3)H(n,k-1).$$
Hence $H(n,k)$ satisfies the same recurrence and initial conditions as $B(n,k)$, so they agree.
\end{proof}

Let $(Dy)^n(z)=zf_n(y)$. Using $(Dy)^{n+1}(z)=D(yzf_n(y))$, we get
\begin{equation}\label{Rny-recu}
f_{n+1}(y)=(1+2y^2)f_n(y)+y(1+y^2)f'_n(y)
\end{equation}
with initial values $f_0(y)=1$.
The first few terms of $f_n(y)$ are given as follows:
\begin{align*}
f_1(y)&=1+2y^2,\\
f_2(y)&=1+8y^2+8y^4,\\
f_3(y)&=1+26y^2+72y^4+48y^6,\\
f_4(y)&=1+80y^2+464y^4+768y^6+384y^8.
\end{align*}

Set $f_n(y)=\sum_{k=0}^nf({n,k})y^{2k}$. By~(\ref{Rny-recu}), we obtain
\begin{equation}\label{fnk-recu}
f(n+1,k)=(1+2k)f(n,k)+2kf(n,k-1)\quad\textrm{for $0\leq k\leq n$,}
\end{equation}
with initial conditions $f(0,0)=1,f(0,k)=0$ for $k\geq 1$.
It should be noted~\cite[A145901]{Sloane} that $$(f(n,0),f(n,1),\ldots,f(n,n))$$
is the $f$-vectors of the simplicial complexes dual to the
permutohedra of type $B$.

\section{Polynomials related to $(yD)^n(y)$ and $(yD)^n(z)$}\label{Section-5}
For $n\geq 1$, we define
\begin{equation}\label{def-derivative}
(yD)^n(y)=\sum_{k=1}^{n}M({n,k})y^{2k-1}z^{2n-2k+2}\quad {\text and}\quad (yD)^n(z)=\sum_{k=1}^{n}N({n,k})y^{2k}z^{2n-2k+1}.
\end{equation}
\begin{theorem}
For $1\leq k\leq n$, we have
\begin{equation}\label{recurrence-1}
M({n+1,k})=(2k-1)M({n,k})+(2n-2k+4)M({n,k-1}),
\end{equation}
\begin{equation}\label{recurrence-11}
N({n+1,k})=2kN({n,k})+(2n-2k+3)N({n,k-1}).
\end{equation}
\end{theorem}
\begin{proof}
Note that
$$(yD)^{n+1}(y)=(yD)(yD)^n(y)=\sum_{k=1}^n(2k-1)M({n,k})y^{2k-1}z^{2n-2k+4}+
\sum_{k=1}^n(2n-2k+2)M({n,k})y^{2k+1}z^{2n-2k+2}.$$
Thus we obtain~(\ref{recurrence-1}).
Similarly, we get~(\ref{recurrence-11}).
\end{proof}

From~(\ref{recurrence-1}) and~(\ref{recurrence-11}), we immediately get a connection between $M({n,k})$ and $N({n,k})$.
\begin{corollary}
For $1\leq k\leq n$, we have
$M({n,k})=N({n,n-k+1})$.
\end{corollary}

Let
\begin{equation*}
M_n(x)=\sum_{k=1}^nM({n,k})x^k\quad {\text and}\quad N_n(x)=\sum_{k=1}^nN({n,k})x^k.
\end{equation*}
Then
\begin{equation}\label{MnxNnxSymm}
M_n(x)=x^{n+1}N_n\left(\frac{1}{x}\right).
\end{equation}

Let
\begin{equation}\label{Def-2}
R_n(y)=(yD)^n(y)=\sum_{k=0}^nR({n,k})y^{2k+1}\quad {\text and}\quad zT_n(y)=(yD)^n(z)=z\sum_{k=1}^nT({n,k})y^{2k}.
\end{equation}
Using~(\ref{Def-2}), it is easy to verify that
\begin{equation}\label{Rny-recurr}
R_{n+1}(y)=y(1+y^2)R'_n(y)
\end{equation}
and
\begin{equation}\label{Galton}
T_{n+1}(y)=y^2T_n(y)+y(1+y^2)T'_n(y).
\end{equation}
The first few terms of $R_n(y)$ and $T_n(y)$ are given as follows:
\begin{align*}
R_1(y)&=y+y^3,\\
R_2(y)&=y+4y^3+3y^5,\\
R_3(y)&=y+13y^3+27y^5+15y^7,\\
R_4(y)&=y+40y^3+174y^5+240y^7+105y^9,\\
R_5(y)&=y+121y^3+990y^5+2550y^7+2625y^9+945y^{11};\\
T_1(y)&=y^2,\\
T_2(y)&=2y^2+3y^4,\\
T_3(y)&=4y^2+18y^4+15y^6,\\
T_4(y)&=8y^2+84y^4+180y^6+105y^8\\
T_5(y)&=16y^2+360y^4+1500y^6+2100y^8+945y^{10}.
\end{align*}
Equating the coefficient of $y^{2k+1}$ on both sides of~(\ref{Rny-recurr}), we get
\begin{equation*}\label{Rny-recu-2}
R({n+1,k})=(2k+1)R({n,k})+(2k-1)R({n,k-1}).
\end{equation*}
Equating the coefficient of $y^{2k}$ on both sides of~(\ref{Galton}), we get
\begin{equation*}\label{GaltonTri}
T({n+1,k})=2kT({n,k})+(2k-1)T({n,k-1}).
\end{equation*}
Clearly, $R({n,n})=T({n,n})=(2n-1)!!$, where $(2n-1)!!$ is the {\it double factorial number}.
It should be noted that the triangular arrays $\{R({n,k})\}_{n\geq 1,0\leq k\leq n}$ and $\{T({n,k})\}_{n\geq 1,1\leq k\leq n}$ are both {\it Galton triangles} (see~\cite[\textsf{A187075 }]{Sloane} for instance), and it has been studied by Neuwirth~\cite{Neuwirth01}.
We now present the following result.
\begin{theorem}\label{RnyMny}
For $n\geq 1$, we have
\begin{equation*}
R_n(y)=y^{2n+1}N_n\left(\frac{1+y^2}{y^2}\right)\quad {\text and}\quad T_n(y)=
(1+y^2)^{n}N_n\left(\frac{y^2}{1+y^2}\right).
\end{equation*}
\end{theorem}
\begin{proof}
Note that $z^2=y^2+1$. Combining~(\ref{def-derivative}) and~(\ref{Def-2}), we obtain
$$R_n(y)=\sum_{k=1}^{n}M({n,k})y^{2k-1}(y^2+1)^{n-k+1}=
y^{-1}(1+y^2)^{n+1}M_n\left(\frac{y^2}{1+y^2}\right),$$
and
$$T_n(y)=\sum_{k=1}^{n}N({n,k})y^{2k}(y^2+1)^{n-k}=
(1+y^2)^{n}N_n\left(\frac{y^2}{1+y^2}\right).$$
Using~(\ref{MnxNnxSymm}), we get
$$(1+y^2)^{n+1}M_n\left(\frac{y^2}{1+y^2}\right)=y^{2n+2}N_n\left(\frac{1+y^2}{y^2}\right)$$
as desired.
\end{proof}

By Theorem~\ref{RnyMny}, we get $R_n(1)=N_n(2)$ and $T_n(1)=2^nN_n(\frac{1}{2})$.
It follows from~(\ref{recurrence-11}) that
\begin{equation}\label{recu-Nnx}
N_{n+1}(x)=(2n+1)xN_n(x)+2x(1-x)N'_n(x)
\end{equation}
with initial values $N_0(x)=1$. The first few terms of $N_n(x)$ can be computed directly as follows:
\begin{align*}
N_1(x)&=x,\\
N_2(x)&=2x+x^2,\\
N_3(x)&=4x+10x^2+x^3,\\
N_4(x)&=8x+60x^2+36x^3+x^4,\\
N_5(x)&=16x+296x^2+516x^3+116x^4+x^5.
\end{align*}
In particular, $N({n,1})=2^{n-1}$, $N({n,n})=1$ and $N_n(1)=(2n-1)!!$ for $n\geq 1$.
In the following discussion, we will study some properties of $N_n(x)$.

The numbers $N({n,k})$ arise often in combinatorics and other branches of mathematics (see~\cite{Lehmer85} for instance). A {\it perfect matching} of $[2n]$ is a partition of $[2n]$ into $n$ blocks of size $2$. Using~(\ref{recurrence-11}) and analyzing the placement of $2n-1$ and $2n$, it is easy to verify that the number $N({n,k})$ counts perfect matchings of $[2n]$ with the restriction that only $k$ matching pairs have odd smaller entries (see~\cite[\textsf{A185411}]{Sloane}).
It is well known~\cite[\textsf{A156919}]{Sloane} that
\begin{equation}\label{Explicit}
N_n(x)=\sum_{k=1}^n2^{n-2k}\binom{2k}{k}k!S(n,k)x^k(1-x)^{n-k} \quad {\text for}\quad n\geq 1,
\end{equation}
where $S(n,k)$ is {\it the Stirling number of the second kind}.
By~(\ref{Explicit}), we get
$$N({n,k})=\sum_{i=1}^k(-1)^{k-i}2^{n-2i}\binom{2i}{i}\binom{n-i}{k-i}i!S(n,i).$$

Let
$$N(x,t)=\sum_{n\geq 0}N_n(x)\frac{t^n}{n!}.$$
Using~(\ref{recu-Nnx}), the formal power series $N(x,t)$ satisfies the following
partial differential equation:
\begin{equation*}\label{diff-eq}
(1-2xt)\frac{\partial N(x,t)}{\partial t}-2x(1-x)\frac{\partial N(x,t)}{\partial x}=xN(x,t).
\end{equation*}
By {\it the method of characteristics}~\cite{Wilf}, it is easy to derive an explicit form:
\begin{equation*}
N(x,t)=e^{xt}\sqrt{\frac{1-x}{e^{2xt}-xe^{2t}}}.
\end{equation*}
Hence
\begin{equation}\label{N2xt}
N^2(x,t)=\frac{1-x}{1-xe^{2t(1-x)}}.
\end{equation}
Combining~(\ref{Axz}) and~(\ref{N2xt}), we get the following result.
\begin{theorem}
For $n\geq 0$, we have
\begin{equation*}
\sum_{k=0}^n\binom{n}{k}N_k(x)N_{n-k}(x)=2^nA_n(x).
\end{equation*}
\end{theorem}

In the following of this section, we will give both central and local limit theorems for the coefficients of $N_n(x)$.
As an application of a result~\cite[Theorem 2]{Ma08} on polynomials with only real zeros,
the recurrence relation~(\ref{recu-Nnx}) enables us to show that the polynomials $\{N_n(x)\}_{n\geq 1}$
form a {\it Sturm sequence}.
\begin{proposition}\label{RealZeros}
For $n\geq 2$, the polynomial $N_n(x)$ has $n$ distinct real zeros, separated by the zeros of $N_{n-1}(x)$.
\end{proposition}

Let $\{a(n,k)\}_{0\leq k\leq n}$ be a sequence of positive real numbers. It has no {\it internal zeros} if and only if there exist no indices $i<j<k$ with $a(n,i)a(n,k)\neq0$ but $a(n,j)=0$.
Let $A_n=\sum_{k=0}^na(n,k)$. We say the sequence $\{a(n,k)\}$ satisfies a central limit theorem with mean $\mu_n$ and variance $\sigma_n^2$ provided
$$\limsup_{n\rightarrow+\infty,x\in\mathbb{R}}\left|\sum_{k=0}^{\mu_n+x\sigma_n}\frac{a(n,k)}{A_n} -\frac{1}{\sqrt{2\pi}}\int_{-\infty}^xe^{-\frac{t^2}{2}}dt\right|=0.$$
The sequence satisfies a local limit theorem on $B\in\mathbb{R}$ if
$$\limsup_{n\rightarrow+\infty,x\in B}\left|\frac{\sigma_na(n,\mu_n+x\sigma_n)}{A_n} -\frac{1}{\sqrt{2\pi}}e^{-\frac{x^2}{2}}\right|=0.$$
Recall the following Bender's theorem.
\begin{theorem}\cite{Bender73}\label{bender}
Let $\{P_n\}_{n\geq1}$ be a sequence of polynomials with only real zeros. The sequence of the coefficients of $P_n$ satisfies a central limit theorem with $$\mu_n=\frac{P_n'(1)}{P_n(1)} \quad\textrm{and}\quad
\sigma_n^2=\frac{P_n'(1)}{P_n(1)}+\frac{P_n''(1)}{P_n(1)}-\left(\frac{P_n'(1)}{P_n(1)}\right)^2,$$
provided that $\lim\limits_{n\to\infty}\sigma_n^2=+\infty$.
If the sequence of coefficients of each $P_n(x)$ has no internal zeros, then the sequence of coefficients satisfies a local limit theorem.
\end{theorem}

Combining Proposition~\ref{RealZeros} and Theorem~\ref{bender}, we obtain the following result.
\begin{theorem}\label{mthm-2}
The sequence $\{N({n,k})\}_{1\leq k\leq n}$ satisfies a central and a local limit theorem with
$\mu_n={(2n+1)}/{4}$ and $\sigma_n^2=(2n+1)/24,$ where $n\geq 4$.
\end{theorem}
\begin{proof}
By differentiating~(\ref{recu-Nnx}), we obtain the recurrence
$x_{n+1}=(2n+1)!!+(2n-1)x_n$
for $x_n=N_n'(1)$, and this has the solution
$x_n=(2n+1)!!/4$ for $n\geq 2
$. By Theorem~\ref{bender}, we have $\mu_n={(2n+1)}/{4}$.
Another differentiation leads to the recurrence
$$y_{n+1}=\frac{(2n+1)!!}{4}(4n-2)+(2n-3)y_n$$
for $y_n=N_n''(1)$. Set $y_n=(2n-1)!!(an^2+bn+c)$ and solve for $a,b,c$ to get
$$y_n={(2n-1)!!}(12n^2-8n-7)/{48}$$
for $n\geq 4$.
Hence $\sigma_n^2=(2n+1)/24$.
Thus $\lim\limits_{n\to\infty}\sigma_n^2=+\infty$ as desired.
\end{proof}

Let $P(x)=\sum_{i=0}^na_ix^i$ be a polynomial. Let $m$ be an index such that $a_m=\max_{0\leq i\leq n}a_i$.
Darroch~\cite{Darroch64} showed that if $P(x)\in\rz(-\infty,0]$, then
$$\left\lfloor{\frac{P_n'(1)}{P_n(1)}}\right\rfloor\leq m\leq \left\lceil{\frac{P_n'(1)}{P_n(1)}}\right\rceil.$$
So the following result is immediate.
\begin{corollary}
Let $i={\lrf{(2n+1)}/{4}}$ or $i={\lrc{(2n+1)}/{4}}$. Then $N({n,i})=\max_{1\leq k\leq n}N({n,k})$.
\end{corollary}


\end{document}